\documentclass[10pt]{article}
\usepackage{mathrsfs}
\usepackage{amsthm}
\usepackage{amssymb}
\usepackage{enumerate}
\usepackage{amsmath}
\usepackage{graphicx}
\usepackage{amsfonts}
\usepackage{float}
\usepackage{cite}
\usepackage[text={150mm,220mm}]{geometry}
\usepackage[colorlinks,
            linkcolor=blue,
            anchorcolor=blue,
            citecolor=blue
            ]{hyperref}
\usepackage{xcolor}
\usepackage{comment}

\usepackage{titlesec}
\titleformat{\subsection}[hang]{\it}{\thesubsection}{4 pt}{ \it}[]

\newtheorem{theorem}{Theorem}[section]
\newtheorem{remark}[theorem]{Remark}
\newtheorem{lemma}[theorem]{Lemma}
\newtheorem{proposition}[theorem]{Proposition}

\newtheorem*{claim}{Claim}
\newtheorem*{initial}{Initial Condition}
\numberwithin{equation}{section}
\normalsize

\usepackage{titlesec}
\titleformat*{\section}{\centering}
\titleformat{\subsection}[runin]{\bf}{\thesubsection}{4 pt}{ \bf}[]

\title{Non-equilibrium Fluctuations of the Weakly Asymmetric Normalized Binary Contact Path Process}
\author{Xiaofeng Xue \thanks{\textbf{E-mail}: xfxue@bjtu.edu.cn \textbf{Address}: School of Science, Beijing Jiaotong University, Beijing 100044, China.} and { Linjie Zhao} \thanks{\textbf{E-mail}: zhaolinjie@pku.edu.cn \textbf{Address}: Inria Lille-Nord Europe, France.}}
\date{}

\begin{document}

\maketitle

\begin{abstract}
This paper is a further investigation of the problem studied in
\cite{xue2020hydrodynamics}, where the authors proved a law of large numbers for the empirical measure of the weakly asymmetric normalized binary contact path process on $\mathbb{Z}^d,\, d \geq 3$, and then conjectured that a central limit theorem should hold under a non-equilibrium initial condition.  We prove that the aforesaid conjecture is true when the dimension $d$ of the underlying lattice and the infection rate $\lambda$ of the process are sufficiently large.

\quad

\end{abstract}

\section{Introduction}

\subsection{Background.}

We prove a central limit theorem for the empirical measure of the weakly asymmetric normalized binary contact path process, whose law of large numbers was derived in \cite{xue2020hydrodynamics}. Since we assume the initial distribution  is non-stationary, the fluctuations are non-equilibrium. Non-equilibrium fluctuations have long been an open problem in the theory of hydrodynamic limit for interacting particle systems. Previous results such as in \cite{ravishankar1992fluctuations,PresuttiSpohn83} heavily depend on the special structure of the concerned model, such as self-duality.  Recently, Erhard {\it et al.} \cite{erhard2020non} considered non-equilibrium fluctuations for the symmetric simple exclusion process with a slow bond by  deriving a precise estimate on the correlation function of the system. Moreover, Jara and Menezes  \cite{jara2018non,jaram18nonequilireaction} introduced a powerful tool to obtain non-equilibrium fluctuations of interacting particle systems. The  proof is mainly based on a sharp estimate on the relative entropy \cite{yau1991relative} of the law of the process with respect to the product reference measure associated to the hydrodynamic limit profile. Such an approach was also investigated by Jara and Landim \cite{jara2020stochastic} in the context of a stirring process perturbed by a voter model.

Informally, the model  considered in this paper can be interpreted as the infection of a disease on the lattice $\mathbb{Z}^d$.  To each site $x$ is associated a value $\eta (x) \in [0,\infty)$, which describes the seriousness of the disease at site $x$.  At each transition time,  site $x$ could be recovered or infected by its neighbors at some given rate. Between transition times, the value $\eta (x)$ evolves according to some ODE. We refer the readers to the next subsection for a detailed and rigorous description of the model. Main features of the process are listed below. First of all, the configuration can take arbitrarily large value at each site. Second, at each transition time, a large amount of quantity could be added to or removed from a site. These  features raise two main problems, as already addressed in \cite{xue2020hydrodynamics}, in the proof of fluctuations. The first is to obtain a uniform bound for  the fourth moment of the occupation  variable. The second is to prove continuity of the limiting trajectory.

It is well-known that the model considered in this paper belongs to a larger class of interacting particle systems called linear systems (Cf. \cite[Chapter 9]{liggettips}).  In such a system, the usual way to bound the moments of the occupation variable is first to use the Hille-Yosida Theorem and then to find the positive eigenvector corresponding to the eigenvalue zero of a matrix obtained from the Hille-Yosida Theorem. When bounding the fourth moment, such a strategy is almost impossible since we have to consider a  $(\mathbb{Z}^d)^4\times (\mathbb{Z}^d)^4$ matrix. Instead, depending on the special structure of the model, we relate the fourth moment to a random walk  and use random walk techniques to bound the fourth moment. We believe that such a method could be used to bound higher moments of occupation random variables. However, the main drawback of this method is that we need the dimension of the underlying lattice and the infection rate of the model to be large enough.

\subsection{The model.}\label{the model} In this subsection, we introduce the model formally. The  weakly asymmetric  normalized binary contact path process on the lattice $\mathbb{Z}^d$ is a continuous-time Markov process with state space $[0,+\infty)^{\mathbb{Z}^d}$. Any element $\eta \in [0,+\infty)^{\mathbb{Z}^d}$ is called a {\it configuration}.   For $x,y\in \mathbb{Z}^d$, we write $x\sim y$ if $|x-y| = \sum_{i=1}^d |x_i - y_i| = 1$. We denote by $O$ the origin of $\mathbb{Z}^d$.
For any configuration $\eta\in [0,+\infty)^{\mathbb{Z}^d}$ and $x\in \mathbb{Z}^d$, we define $\eta^x\in [0,+\infty)^{\mathbb{Z}^d}$
as
\[
\eta^x(u)=
\begin{cases}
\eta(u) & \text{~if~}u\neq x,\\
0 & \text{~if~}u=x.
\end{cases}
\]
Moreover, for $x,y\in \mathbb{Z}^d$ such that $x\sim y$, we define $\eta^{x,y}\in [0,+\infty)^{\mathbb{Z}^d}$ as
\[
\eta^{x,y}(u)=
\begin{cases}
\eta(u) & \text{~if~}u\neq x,\\
\eta(x)+\eta(y) & \text{~if~}u=x.
\end{cases}
\]
Fix constants $\lambda,\,\lambda_1,\,\lambda_2 \geq 0 $ and  an integer $N$ large enough such that $\lambda  \geq \lambda_1 / N$. For each $x\in \mathbb{Z}^d$,  let $\{Y_x(t)\}_{t\geq 0}$ be a Poisson process with rate $1$.  Moreover, for each $1 \leq i \leq d$,   let $\{U_{x, x + e_i}(t)\}_{t\geq 0}$ be a Poisson process with rate $\lambda - \lambda_1/N$ and let $\{U_{x,x-{e_i}}(t)\}_{t\geq 0}$ be a Poisson process with rate $\lambda + \lambda_1/N$, where $\{e_i\}_{1 \leq i \leq d}$ is the canonical basis of $\mathbb{Z}^d$.  Note that $U_{x,x+e_i} \neq U_{x+e_i,x}$. Assume that all these Poisson processes are independent. Then the  weakly asymmetric normalized binary contact path process evolves as follows:
\begin{enumerate}[(i)]
	\item  for any event moment $t$ of $Y_x(\cdot)$, $\eta_t=\eta_{t-}^x$, where $\eta_{t-}=\lim_{s<t,s\rightarrow t}\eta_s$,
	
	\item for any event moment $r$ of $U_{x,y}(\cdot)$ where $y \sim x$,\, $\eta_r=\eta_{r-}^{x,y}$,
	
	\item  for $0\leq t_1<t_2$, if there is no event moment of $Y_x(\cdot)$ or $\{U_{x,y}(\cdot):~y\sim x\}$ in $[t_1,t_2]$, then
	\[
	\eta_s(x)=\eta_{t_1}(x) \exp \left\{ \left(1-2\lambda d + \lambda_2/N^2 \right)(s-t_1) \right\}, \quad t_1\leq s\leq t_2,
	\]
	or equivalently,
	\[
	\frac{d}{ds}\eta_s(x)=(1-2\lambda d + \lambda_2/N^2)\, \eta_s(x), \quad t_1\leq s\leq t_2.
	\]
\end{enumerate}

The  weakly asymmetric  normalized binary contact path process can also be defined equivalently via its infinitesimal generator $\mathcal{L}_N$.  More precisely, fix a strictly  positive function  $\alpha: \mathbb{Z}^d \rightarrow \mathbb{R}$ such that
$$
\sum_x \alpha (x)  < \infty,
$$
and let
$$
\mathbf{X} = \left\{ \eta \in [0,\infty)^{\mathbb{Z}^d}: \sum_{x} \alpha (x) \eta (x) < \infty \right\}.
$$
According to the evolution of this process introduced above and   \cite[Theorem 9.1.14]{liggettips}, the generator $\mathcal{L}_N$ of the process is given by
\begin{equation}\label{equ 1.1 generator}
\begin{aligned}
\mathcal{L}_N f(\eta) &=\sum_{x\in \mathbb{Z}^d}\big[f(\eta^x)-f(\eta)\big] + \left(\lambda \pm \frac{\lambda_1}{N}\right)\sum_{x\in \mathbb{Z}^d}\sum_{i=1}^d \big[f(\eta^{x,x \mp e_i})-f(\eta)\big]\\
& \qquad \qquad +\sum_{x\in \mathbb{Z}^d}  \left(1-2\lambda d  + \frac{\lambda_2}{N^2} \right) f_x (\eta)\eta(x),
\end{aligned}
\end{equation}
where $f_x (\eta)$ is the
partial derivative of $f$ with respect to the coordinate $\eta(x)$ and $f$ is taken over all functions on $\mathbf{X}$ such that $f$ has continuous first derivatives $f_x$ and
$$
\sup_x \frac{||f_x||_\infty}{\alpha (x)} < \infty,
$$
where $||\cdot||_\infty$ is the supremum norm of a function.  Let $\{\eta_t\}_{t \geq 0}$ be the process with generator $\mathcal{L}_N$.\footnote{The process $\{\eta_t\}$ actually depends on the scaling parameter $N$ since the generator does. However, to distinguish between the notations $\{\eta_t\}$ and $\{\eta^N_t\}$ introduced soon afterwards and to make notations simple, we suppress the dependence on $N$ here.}  For each $N$, let $\eta^N_t = \eta_{tN^2}$. Then the process $\{ \eta^N_t\}_{t \geq 0}$ has generator $N^2 \mathcal{L}_N$.  Let $\mathbb{P} := \mathbb{P}^N$ be the probability measure on $D([0,\infty),[0,\infty)^{\mathbb{Z}^d})$ induced by the process $\{\eta^N_t\}_{t \geq 0}$ and let $\mathbb{E} := \mathbb{E}^N$ be the corresponding expectation.  Without ambiguity, $\mathbb{P}$ and $\mathbb{E}$ will also be used to denote the law and the expectation of the process $\{\eta_t\}_{t \geq 0}$.

For more background on the model, such as its close relationship with the contact process, we refer the readers to \cite{xue2020hydrodynamics} and references therein.

\subsection{Hydrodynamics.} In this subsection, we review the law of large numbers for the empirical measure proved in \cite{xue2020hydrodynamics}.   Denote by $C_c (\mathbb{R}^d)$  the set of  continuous functions $f : \mathbb{R}^d \rightarrow \mathbb{R}$ with compact support and by $C_c^k (\mathbb{R}^d),\, k \in \mathbb{N}$, the subset of $C_c (\mathbb{R}^d)$ containing all functions furthermore having continuous $k$-th derivatives. Let $\mathcal{M}_+ (\mathbb{R}^d)$ be the set of positive Radon measures in $\mathbb{R}^d$ endowed with the vague topology, i.e., for measures $\mu_n$, $n \geq 1$, and $\mu$ in $\mathcal{M}_+ (\mathbb{R}^d)$, $\mu_n \rightarrow \mu$ as $n \rightarrow \infty$ if and only if for every test function $G \in C_c (\mathbb{R}^d)$, $\langle\mu_n,G\rangle \rightarrow \langle\mu,G\rangle$. Here $\langle \mu, G\rangle = \int G\, d \mu$ for $\mu \in \mathcal{M}_+ (\mathbb{R}^d)$ and $G \in C_c (\mathbb{R}^d)$.
For any $t \geq 0$, let $\pi^N_t (d u) \in \mathcal{M}_+ (\mathbb{R}^d)$ be the random empirical measure given by
\[
\pi^N_t (du):= \frac{1}{N^d}\sum_{x\in \mathbb{Z}^d}\eta_{t }^N(x)\,\delta_{x/N}(du),
\]
where $\delta_{x/N}$ is the usual Dirac measure concentrated on the point $x/N$.

We say $\rho (t,u)$ is a weak solution of the following linear parabolic equation
\begin{equation}\label{eq:intro1}
\begin{cases}
\partial_t \rho \,(t,u) = \lambda\, \Delta \rho\, (t,u) - 2 \lambda_1 \sum_{i=1}^d \partial_{u_i} \rho\, (t,u) + \lambda_2 \, \rho\, (t,u) , \quad t > 0,\,u \in \mathbb{R}^d,\\
\rho (0,u) = \rho_0 (u), \quad  u \in \mathbb{R}^d,
\end{cases}
\end{equation}
if for each $t$, $\rho (t,\cdot)$ is  integrable and  for any $G \in \mathcal{S} \left( \mathbb{R}^d \right)$ and any $t > 0$,
\begin{equation*}
\left \langle \rho (t,\cdot), G\right \rangle - \left \langle \rho_0 (\cdot), G\right \rangle
=  \int_{ 0 }^t d s \Big\{ \lambda \left \langle \rho (s,\cdot), \Delta G \right \rangle   +  2 \lambda_1 \sum_{i=1}^d \left \langle \rho (s,\cdot), \partial_{u_i} G \right \rangle +  \lambda_2\,\left \langle \rho (s,\cdot), \,G\right \rangle
\Big\},
\end{equation*}
where $\rho_0$ is the initial density profile and $\mathcal{S} \left( \mathbb{R}^d\right)$ is the family of all Schwartz functions.

\,

\begin{theorem}[{\cite[Theorem 1.1]{xue2020hydrodynamics}}]\label{thm hl}
	Let $\rho_0: \mathbb{R}^d \rightarrow [0,\infty)$ be bounded and integrable. Initially, $\{\eta^N_0(x)\}_{x\in \mathbb{Z}^d}$ are independent and
	$
	\mathbb{E} \left[ \eta^N_0 (x) \right]= \rho_0 (x/N)
	$ for all $x\in \mathbb{Z}^d$.
	Moreover, assume the second moment is uniformly bounded,
\[
\sup_{x\in \mathbb{Z}^d,N\geq 1} \mathbb{E}\big[\eta^N_0(x)^2\big]<\infty.
\]
	 Suppose $d \geq 3$ and
	\[
	\lambda>\frac{1}{2d(2\gamma_d-1)},
	\]
	where $\gamma_d$ is the escape probability that the simple random walk on $\mathbb{Z}^d$ starting at $O$ never returns to $O$ again. Then for all $t \geq 0$, as $N \rightarrow \infty$,
	\begin{equation*}
	\pi^N_t (d u) \rightarrow \rho (t, u) du~~\text{in probability,}
	\end{equation*}
	where $\rho (t,u)$ is the unique weak solution to the  equation
	\eqref{eq:intro1}.
\end{theorem}

\subsection{Fluctuations.}  In this subsection, we state our main result  for the  central limit theorem of the empirical measure. Fix hereafter a horizontal time $T > 0$. Let  $\mathcal{S}' (\mathbb{R}^d)$ be the space of tempered distributions.  Let $D \left( [0,T], \mathcal{S}' (\mathbb{R}^d) \right)$ be the space of $\mathcal{S}' (\mathbb{R}^d)$-valued trajectories which are right continuous and have left limits. The density fluctuations field $ \mathcal{Y}_t^N$, $0 \leq t \leq T$, is an $\mathcal{S}' (\mathbb{R}^d)$-valued process defined as
\begin{equation}\label{fluceqn1}
\mathcal{Y}_t^N (H) = \frac{1}{N^{1+d/2}} \sum_{x \in \mathbb{Z}^d} \bar{\eta}_{t}^N (x) \,H \left(\frac{x}{N}\right),
\end{equation}
where $H \in \mathcal{S} (\mathbb{R}^d)$ and $\bar{\eta}_t^N (x)$ is the centered occupation variable $\bar{\eta}_t^N (x) = \eta_t^N (x) - \mathbb{E} \left[\eta_t^N (x)\right]$.  Note that the space is divided by $N^{1+d/2}$ instead of the usual scaling $N^{d/2}$. We shall prove a central limit theorem under the following non-equilibrium initial condition.

\,

\begin{initial}
Suppose that $\{\eta^N_0(x)\}_{x\in \mathbb{Z}^d}$ are independent and identically distributed random variables with $\mathbb{E} \left[ \eta^N_0(x) \right] \equiv 1$. Furthermore, assume $\eta^N_0(x)$ has  bounded fourth moment uniformly in $N$,
\begin{equation}\label{square-bounded2}
\sup_{N\geq 1}\mathbb{E}\big[\eta^N_0(x)^4\big] < \infty.
\end{equation}
\end{initial}

\,

We remark that the constant one appearing in the first moment above  is not  important at all. Now we are ready to state our main result.

\begin{theorem}\label{main}
Take $\lambda_2 = 0$. There exist $d_0$ and $\lambda_0$ such that for any $d > d_0$ and for any $\lambda > \lambda_0$, the sequence of the processes $\{\mathcal{Y}_t^N, 0 \leq t \leq T\}_{N \geq 1}$ converges in distribution, as $N \rightarrow \infty$, with respect to the Skorohod topology of $D \left( [0,T], \mathcal{S}' (\mathbb{R}^d) \right) $ to the generalized Ornstein-Uhlenbeck process $\{\mathcal{Y}_t, 0 \leq t \leq T\}$ in $C \left( [0,T], \mathcal{S}' (\mathbb{R}^d) \right)$, the formal solution of
	\begin{equation}\label{OU}
	\begin{cases}
	& d \mathcal{Y}_t = \Big(\lambda\Delta \mathcal{Y}_t-2\lambda_1\sum_{i=1}^d\partial_{u_i}\mathcal{Y}_t\Big)dt + \sqrt{C (\lambda,d)} \,d \mathcal{W}_t,\\
	& \mathcal{Y}_0=0,
	\end{cases}
	\end{equation}
	where $C (\lambda,d) = (1 + 2 \lambda d )\left(1 + 1/h_\lambda\right)$ with  $h_\lambda$ given by
	\begin{equation}\label{eqn1}
	h_{\lambda}=\frac{2 \lambda d\left(2 \gamma_{d}-1\right)-1}{1+2 d \lambda}
	\end{equation}
	 and $\mathcal{W}_t$ is a space time white noise of unit variance. Recall $\gamma_d$ is given in Theorem \ref{thm hl}.
\end{theorem}

\begin{remark}\label{rm3}
Based on \cite{HolleyStroock78}, the rigorous meaning of $\mathcal{Y}_\cdot$ is the unique random element taking values in the space $C \left( [0,T], \mathcal{S}' (\mathbb{R}^d) \right) $ such that:
\begin{enumerate}[(i)]
	\item For every function $H \in \mathcal{S} (\mathbb{R}^d)$ and every  $G \in C_c^\infty (\mathbb{R})$,
	\begin{align*}
	G (\mathcal{Y}_t (H)) - G (\mathcal{Y}_0 (H))
	&- \int_{ 0 }^t  G' (\mathcal{Y}_s (H)) \Big( \lambda  \mathcal{Y}_s (\Delta H) + 2 \lambda_1 \sum_{i=1}^d  \mathcal{Y}_s (\partial_{u_i} H) \Big)\,ds\\
	&\qquad \qquad - \frac{1}{2} C(\lambda,d) ||H||_2^2 \int_{ 0 }^t G'' ( \mathcal{Y}_s (H))\,ds
	\end{align*}
	is a martingale with respect to the natural filtration $\mathcal{F}_t = \sigma (\mathcal{Y}_s (H), s \leq t, H \in \mathcal{S} (\mathbb{R}^d))$, where $||H||_2$ is the $L^2$-norm of the function $H$, i.e., $||H||_2^2=\int_{\mathbb{R}^d}H(u)^2\,du$.
	\item $\mathcal{Y}_0$ is identically equal to zero: $\mathcal{Y}_0 \equiv 0$.
\end{enumerate}
\end{remark}

\begin{remark}
As observed in \cite[Remark 1.3]{xue2020hydrodynamics},  we could reduce the case $\lambda_2 \neq 0$ to the case $\lambda_2 = 0$.  Indeed, if $\lambda_2 \neq 0$, then the sequence of the processes
$
\{\mathcal{Y}_t^N, 0 \leq t \leq T\}_{N \geq 1}
$
converges  in distribution, as $N \rightarrow \infty$, with respect to the Skorohod topology of $D \left( [0,T], \mathcal{S}' (\mathbb{R}^d) \right) $ to
$\{e^{\lambda_2t}\mathcal{Y}_t, 0\leq t\leq T\}$, where \{$\mathcal{Y}_t, 0\leq t\leq T\}$ is the unique solution to \eqref{OU}. By Ito's formula, informally,  $\{e^{\lambda_2t}\mathcal{Y}_t, 0\leq t\leq T\}$ can also be interpreted as the  solution to the SDE
\begin{equation*}
\begin{cases}
& d \mathcal{Z}_t = \Big(\lambda\Delta \mathcal{Z}_t-2\lambda_1\sum_{i=1}^d\partial_{u_i}\mathcal{Z}_t+\lambda_2\mathcal{Z}_t\Big)dt + e^{\lambda_2t}\sqrt{C (\lambda,d)} d \mathcal{W}_t,\\
& \mathcal{Z}_0=0.
\end{cases}
\end{equation*}
That's why we only discuss the case where $\lambda_2=0$.
\end{remark}

\,

The rest of the paper is organized as follows. In Section \ref{sec2} we give a detailed proof of Theorem \ref{main}, except two main propositions, i.e., Propositions \ref{fourth moment} and \ref{correlation}. These two propositions are crucial  and their proofs are postponed to Sections \ref{sec fourm} and \ref{sec cor}.

\section{Proof of Theorem \ref{main}}\label{sec2}

In this section, we prove Theorem \ref{main}.  The strategy to prove such a result is routine. We first prove the tightness of the sequence $\{\mathcal{Y}^N_t,\,0 \leq t \leq T\}_{N \geq 1}$ by checking Aldous' criterion. Then we prove the uniqueness of the limit following the martingale approach. Since the general case $\lambda_1 > 0$ follows from the same strategy, to fix ideas, we shall only present the proof for the case $\lambda_1  = 0$.  Note that most of our results hold for  $d\geq 3$ and $\lambda>\frac{1}{2d(2\gamma_d-1)}$ except in Propositions \ref{fourth moment} and \ref{correlation}, where  we need $d$ and $\lambda$ to be large enough.

Since  Propositions \ref{fourth moment} and \ref{correlation} are one of the main contributions of this paper, and since they will be used frequently in the proof of Theorem \ref{main}, we state them below first.   Proposition \ref{fourth moment} states that the fourth moment of the occupation variable is uniformly bounded, and Proposition \ref{correlation} says that $\eta^N_t (x)^2$ and $\eta^N_t (y)^2$ are asymptotically independent in the limit $N \rightarrow \infty$  when sites $x$ and $y$ are far away. Their  proofs are postponed to Sections \ref{sec fourm} and \ref{sec cor}.

\begin{proposition}\label{fourth moment}
	There exist $d_0$ and $\lambda_0$ such that for any $d > d_0$ and any $\lambda > \lambda_0$,
	\begin{equation*}
	\sup_{0 \leq t \leq T,\, N \geq 1}\,\mathbb{E}\, [\eta_t^N(x)^4] < \infty \quad \text{for all $x \in \mathbb{Z}^d$}.
	\end{equation*}
\end{proposition}

\begin{proposition}\label{correlation}
	There exist $d_0$ and $\lambda_0$ such that for any $d > d_0$ and any $\lambda > \lambda_0$,
	\begin{equation*}
	\lim_{N \rightarrow \infty} \,\sup_{|x-y| > N^{1/2}}\, \sup_{0 \leq t \leq T}\,
	\mathrm{Cov}\, \big(\eta_t^N(x)^2,\,\eta_t^N (y)^2\big) =0.
	\end{equation*}
\end{proposition}

\begin{remark}\label{rm mcor}
	Recall that $\eta^N_t = \eta_{tN^2}$.  To make notations short, in Sections \ref{sec fourm} and \ref{sec cor}, we actually shall prove the following stronger results,
	\begin{align*}
	\sup_{0 \leq t < \infty}\,\mathbb{E}\, [\eta_t (x)^4] < \infty, \quad \text{and} \quad
	\lim_{N \rightarrow \infty} \,\sup_{|x-y| > N^{1/2}}\, \sup_{0 \leq t < \infty}\,
	\mathrm{Cov}\, \big(\eta_t (x)^2,\,\eta_t (y)^2\big) =0.
	\end{align*}
\end{remark}

\subsection{Tightness.}\label{subsec tight}   In this subsection, we prove the tightness of the sequence $\{\mathcal{Y}^N_t,\,0 \leq t \leq T\}_{N \geq 1}$. By Mitoma's Theorem  \cite{mitoma1983tightness}, it suffices to prove that for any $H \in \mathcal{S} (\mathbb{R}^d)$, the real-valued process $\{\mathcal{Y}^N_t (H),\,0 \leq t \leq T\}$ is tight. We shall prove the statement above by checking Aldous’ criterion (Cf. \cite[Chapter 4]{klscaling}).  More precisely, we need to show that
\begin{enumerate}[(i)]
	\item  for every $0 \leq t \leq T$,
	\begin{equation}\label{aldous1}
	\lim_{M \rightarrow \infty}\, \limsup_{N \rightarrow \infty}  \mathbb{P} \left( |\mathcal{Y}_t^N (H)| > M \right) = 0;
	\end{equation}
	
	\item for every $\delta > 0$,
	\begin{equation}\label{aldous2}
	\lim_{\gamma \rightarrow 0}\, \limsup_{N \rightarrow \infty}\, \sup_{\tau \in \mathfrak{T}_T,\,\theta \leq \gamma}\, \mathbb{P} \left( |\mathcal{Y}_\tau^N (H) - \mathcal{Y}_{\tau + \theta}^N (H)| > \delta \right) = 0,
	\end{equation}
	where $\mathfrak{T}_T$ is the set of stopping times with respect to the natural filtration bounded by $T$.
\end{enumerate}

\,

Before proving tightness, we first state two lemmas, which correspond to Lemmas 4.4 and 4.5 in \cite{xue2020hydrodynamics}. Since the proofs are omitted in \cite{xue2020hydrodynamics}, for completeness, we shall give them here. The first lemma is concerned about the second moment of  the occupation variable.

\begin{lemma}\label{fluclemma2}
	For any $0 < t \leq T$ and for any $x \in \mathbb{R}^d$,
	\begin{equation*}
	\lim_{N \rightarrow \infty}  \mathbb{E} \left[ \eta_{t}^N (x)^2 \right] = 1+1/h_\lambda,
	\end{equation*}
	where $h_\lambda$ is defined in \eqref{eqn1}.
\end{lemma}

\begin{proof}
Since both the initial condition and the process are translation invariant, without loss of generality, we can take $x = O$ the origin.  By \cite[Lemma 2.3]{xue2020hydrodynamics},
\[
\mathbb{E} \left[ \eta_t^N (O)^2 \right]
=\sum_{y\in \mathbb{Z}^d} e^{tN^2\mathbf{\Psi}}(O,y)+e^{tN^2\mathbf{\Psi}}(O,O){\rm Var}\big(\eta_0^N (O)\big),
\]
where
\begin{equation*}
\mathbf{\Psi} (O,y) = \begin{cases}
2 \lambda, &\text{$y \sim O$},\\
1 - 2 \lambda d, &\text{$y = O$},\\
0, &\text{otherwise},
\end{cases}
\text{and if $x \neq O$},\quad \mathbf{\Psi} (x,y) = \begin{cases}
2 \lambda, &\text{$y \sim x$},\\
- 4 \lambda d, &\text{$y = x$},\\
0, &\text{otherwise}.
\end{cases}
\end{equation*}
As in the proof of \cite[Lemma 2.4]{xue2020hydrodynamics},  let  $\Lambda(x)=k(x)+h_\lambda$, $x\in \mathbb{Z}^d$, where $k (x)$ is the probability that a symmetric simple  random walk on $\mathbb{Z}^d$ hits the origin eventually when starting from $x$, then
\begin{equation*}
\Lambda(x)=\sum_{y\in \mathbb{Z}^d}e^{t \mathbf{\Psi}}(x,y)\Lambda(y) \quad \text{for any $t\geq 0$ and $x\in \mathbb{Z}^d$. }
\end{equation*}
Therefore,
	\begin{align}\label{equ 4.2}
	\mathbb{E} \left[ \eta_t^N (O)^2 \right]-\Lambda (O) / h_\lambda
	=e^{tN^2\mathbf{\Psi}}(O,O){\rm Var}\big(\eta_0^N (O)\big)-\sum_{y\in \mathbb{Z}^d}e^{tN^2\mathbf{\Psi}}(O,y)\frac{k(y)}{h_\lambda}.
	\end{align}
	By  \cite[Theorem 2.8.13]{liggettips},
	$
	\lim_{t\rightarrow+\infty}e^{t\Psi}(x,y)=0
	$ for all $x,\,y\in \mathbb{Z}^d$. Moreover, since the second moment of $\eta^N_0 (O)$ is bounded uniformly in $N$, the first term on the right-hand side of \eqref{equ 4.2} converges to zero as $N \rightarrow \infty$. Since $\Lambda (y) \geq h_\lambda$ for all $y$, the second term on the right-hand side of \eqref{equ 4.2} is bounded by
	\begin{equation*}
	\frac{\Lambda (O)}{h_\lambda^2}\,\sup_{|y| > M} k(y)  + \frac{1}{h_\lambda}\sum_{|y| \leq M} e^{tN^2 \mathbf{\Psi}} (O,y)
	\end{equation*}
	 for any $M > 0$. Since $\lim\limits_{M\rightarrow+\infty}\,\sup\limits_{|y|>M}\, k(y)=0$ when $d\geq 3$,  we finish the proof by first letting $N \rightarrow \infty$ and then $M \rightarrow \infty$.
\end{proof}

The next lemma states that the second moment of the density fluctuation field, acting on any test function $H$, is uniformly bounded in $N$ and $t$.

\begin{lemma}\label{lemma 4.3}
	For any $H \in \mathcal{S} (\mathbb{R}^d)$,
	\begin{equation}\label{fluceqn9}
	\sup_{0 \leq t \leq T} \,\sup_{N\geq 1} \, \mathbb{E} \left[  \mathcal{Y}_t^N (H)^2 \right] < \infty.
	\end{equation}
\end{lemma}

\begin{proof}
	According to the definition of $\mathcal{Y}_t^N (H)$,
\[
\mathbb{E} \left[  \mathcal{Y}_t^N (H)^2 \right]=\frac{1}{N^{2+d}}\sum_{x\in \mathbb{Z}^d}\sum_{y\in \mathbb{Z}^d}{\rm Cov}\left(\eta_t^N(x), \eta_t^N(y)\right)H\left(\frac{x}{N}\right)H\left(\frac{y}{N}\right).
\]
An analysis similar with that given in the proof of \cite[Prosition 2.1]{xue2020hydrodynamics} shows that, under our initial condition, there exists $C >0$ independent of $x$ and $y$ such that
\[
{\rm Cov}\left(\eta_t^N(x), \eta_t^N(y)\right)\leq C\, k(x-y), \quad \forall x,\,y.
\]
When $d\geq 3$, it is well-known that there exists a constant\footnote{Throughout the paper, the constant $C$ may be different from line to line.} $C >0$  such that $k(x)\leq C \left\|x\right\|_2^{2-d}$ for all $x$ (Cf.\,\cite{Lawler2010}),  where $\|\cdot\|_2$ is also used to denote the $l^2$-norm of a point since this causes no confusion.  Therefore,
\[
\mathbb{E} \left[  \mathcal{Y}_t^N (H)^2 \right]
\leq \frac{C}{N^{2d}}\sum_{x\in \mathbb{Z}^d}\sum_{y\neq x}\left\|\frac{x}{N}-\frac{y}{N}\right\|^{2-d}_2 \left|H\left(\frac{x}{N}\right)\right|\left|H\left(\frac{y}{N}\right)\right|
+\frac{C}{N^{2+d}}\sum_{x\in \mathbb{Z}^d}H\left(\frac{x}{N}\right)^2,
\]
where $C$ is a constant independent of $N$. Then, for sufficiently large $N$,
\[
\mathbb{E} \left[  \mathcal{Y}_t^N (H)^2 \right]
\leq C\int_{\mathbb{R}^d}\int_{\mathbb{R}^d}\left\|u-v\right\|_2^{2-d}\left|H\left(u\right)\right|\left|H\left(v\right)\right|\,du\,dv
+\frac{C}{N^2}\int_{\mathbb{R}^d}H\left(u\right)^2\,du,
\]
from which Lemma \ref{lemma 4.3} follows  directly.
\end{proof}

\,

Now we are ready to prove the tightness. By Dynkin's martingale, for any $H \in \mathcal{S} (\mathbb{R}^d)$,
\begin{align}
\mathcal{M}_t^N (H) &:= \mathcal{Y}_t^N (H) - \mathcal{Y}_0^N (H) - \int_{0}^{t} N^2 \mathcal{L}_N\mathcal{Y}_s^N (H) \,d s,\label{flucmartingale1}\\
\mathcal{N}_t^N (H) &:= \left[ \mathcal{M}_t^N (H) \right]^2 - \int_{0}^{t} \left\{ N^2 \mathcal{L}_N \mathcal{Y}_s^N (H)^2 - 2 \mathcal{Y}_s^N (H) N^2 \mathcal{L}_N \mathcal{Y}_s^N (H) \right\}\,d s\label{flucmartingale2}
\end{align}
are both martingales with respect to the natural filtration.  Direct calculations  yield that
\begin{equation}\label{fluceqn2}
N^2 \mathcal{L}_N \mathcal{Y}_s^N (H) = \lambda\, \mathcal{Y}_s^N (\Delta_N H),
\end{equation}
where $\Delta_N$ is the discrete Laplacian, $\Delta_N H (x/N) = N^2 [H((x+1)/N) + H((x-1)/N) - 2 H(x/N)]$, and
\begin{equation}\label{fluceqn3}
\begin{aligned}
&N^2 \mathcal{L}_N \mathcal{Y}_s^N (H)^2 - 2 \mathcal{Y}_s^N (H) N^2 \mathcal{L}_N \mathcal{Y}_s^N (H) \\
&\qquad = \frac{1}{N^{d}} \sum_{x \in \mathbb{Z}^d} \eta^N_{s}(x)^2 \left( H \left(\frac{x}{N}\right)^2 + \lambda \, \sum_{i=1}^d H \left(\frac{x \pm e_i}{N}\right)^2\right).
\end{aligned}
\end{equation}
It is easy to see that \eqref{aldous1} follows from the Markov inequality and Lemma \ref{lemma 4.3}.  To check \eqref{aldous2}, by \eqref{flucmartingale1}, we only need to deal with the martingale term and the integral term  respectively. For the martingale term,  by Eq. \eqref{fluceqn3},
\begin{align*}
&\mathbb{E}\big[\left(\mathcal{M}_{\tau+\theta}^{N}(H)-\mathcal{M}_{\tau}^{N}(H)\right)^{2}\big]\\
&=\frac{1}{N^d}\sum_{x\in \mathbb{Z}^d}\left(H\left(\frac{x}{N}\right)^2+\lambda\sum_{i=1}^dH\left(\frac{x\pm e_i}{N}\right)^2\right) \mathbb{E}\Big[ \int_\tau^{\tau+\theta} \eta^N_s(x)^2\,ds \Big]
\end{align*}
Hence, by Chebyshev's inequality, we only need to show that
\begin{equation}\label{equ tightEtoZero}
\lim_{\gamma\rightarrow0}\,\sup_{x\in \mathbb{Z}^d, N\geq 1, \tau\in\mathfrak{T}_T, \theta\leq \gamma}\,\mathbb{E}\Big[\int_\tau^{\tau+\theta}\eta^N_s(x)^2ds\Big]=0.
\end{equation}
By Proposition \ref{fourth moment} and H{\" o}lder's inequality,
\begin{align*}
&\mathbb{E}\Big[ \int_\tau^{\tau+\theta} \eta^N_s(x)^2\,ds\Big]=\mathbb{E}\Big[ \int_0^{T+\theta}\eta^N_s(x)^21_{\{\tau\leq s\leq \tau+\theta\}}\,ds\Big]\\
&=\int_0^{T+\theta} \mathbb{E} \big[  \eta^N_s(x)^21_{\{\tau\leq s\leq \tau+\theta\}} \big] \,ds\\
&\leq \int_0^{T+\theta}\sqrt{\mathbb{E}\big[ \eta_s^N(x)^4\big]}\sqrt{\mathbb{P}\big(\tau\leq s\leq \tau+\theta\big)}\,ds
\leq C\int_0^{T+\theta}\sqrt{\mathbb{P}\left(\tau\leq s\leq \tau+\theta\right)}\,ds\\
&\leq C\sqrt{T+\theta}\sqrt{\int_0^{T+\theta}\mathbb{P}\left(\tau\leq s\leq \tau+\theta\right)ds}
=C\sqrt{T+\theta}\sqrt{\mathbb{E}\Big[ \int_0^{T+\theta}1_{\{\tau\leq s\leq \tau+\theta\}} ds\Big]}\\
&=C\sqrt{T+\theta}\sqrt{\mathbb{E}\Big[ \int_{\tau}^{\tau+\theta}1\,ds\Big]}=C\sqrt{(T+\theta) \theta},
\end{align*}
from which  Eq.\,\eqref{equ tightEtoZero} follows and then the martingale term satisfies \eqref{aldous2}. For the integral term, by \eqref{fluceqn2},
\begin{equation*}
\mathbb{E} \left[ \left(\int_\tau^{\tau+\theta} \mathcal{Y}_s^N {(\Delta_N H)} ds\right)^2 \right] \leq \mathbb{E}  \left[ \theta \int_\tau^{\tau+\theta} \left(\mathcal{Y}_s^N {(\Delta_N H)} \right)^2 ds \right]
\leq     \theta \int_0^{T+\theta} \mathbb{E} \left[  \left(\mathcal{Y}_s^N {(\Delta_N H)} \right)^2 \right] ds.
\end{equation*}
Therefore, the integral term also satisfies \eqref{aldous2} by Lemma \ref{lemma 4.3}. This proves the tightness of the sequence $\{\mathcal{Y}^N_t,\,0 \leq t \leq T\}_{N \geq 1}$.

\subsection{Characterization of the limit.}   Since the sequence $\{\mathcal{Y}^N_t,\,0 \leq t \leq T\}_{N \geq 1}$ is tight, any subsequence further has a subsequence that converges. For short, we still denote the subsequence by $\{\mathcal{Y}^N_t\}$ and denote the limit by $\{\mathcal{Y}_t\}$.  Then the uniqueness of $\{\mathcal{Y}_t\}$ the limit follows directly from Lemmas \ref{lem1} and \ref{lem2} below.

\begin{lemma}\label{lem1}
If  $d$ and  $\lambda$ are large enough, then   for any function $H \in \mathcal{S} (\mathbb{R}^d)$ and for any $\epsilon > 0$,
\begin{equation}
\limsup_{N \rightarrow \infty}\, \mathbb{P} \Big( \sup_{0 \leq t \leq T} |\mathcal{Y}^N_t (H) - \mathcal{Y}^N_{t-} (H)| > \epsilon \Big) = 0.
\end{equation}
\end{lemma}

Since the mapping from $\omega \in D ([0,T],\mathbb{R})$ to $\sup_{0 \leq t \leq T} |\omega_t - \omega_{t-}|$ is continuous, from the above lemma,  the limit  $\{\mathcal{Y}_t\}$   is continuous.

\begin{proof}[Proof of Lemma \ref{lem1}]
	Observe that
	\begin{equation}\label{d4}
	\begin{aligned}
	&\sup_{0 \leq t \leq T} |\mathcal{Y}^N_t (H) - \mathcal{Y}^N_{t-} (H)|^2  \leq \sup_{x\in \mathbb{Z}^d,\, 0 \leq t \leq T}\, N^{-(d+2)}H (x/N)^2 \left( \eta_{t }^N (x) + \sum_{y \sim x} \eta^N_t (y) \right)^2\\
	&\qquad \leq (1 + 2 d) N^{-(d+2)}\,\sup_{0 \leq t \leq T}\, \sum_{x \in \mathbb{Z}^{d}} \eta^N_t (x)^2  \left( H (x/N)^2 + \sum_{y \sim x} H(y/N)^2\right).
	\end{aligned}
	\end{equation}
	
\,
	
	\begin{claim}
		For any $H \in \mathcal{S} (\mathbb{R}^d)$, the process $\big\{ N^{-(d+2)} \sum_{x \in \mathbb{Z}^{d}} \eta^N_t (x)^2  H (x/N) \big\}$ converges in distribution, as $N \rightarrow \infty$, with respect to the Skorohod topology of $D \left( [0,T], \mathbb{R} \right) $ to $\mathbf{0}$, the trajectory identical to zero on $0 \leq t \leq T$.
	\end{claim}

\,

	Since the mapping from $\omega \in D \left( [0,T], \mathbb{R} \right)$ to $\sup_{0 \leq t \leq T}\, \omega_t$ is continuous, the last term in \eqref{d4} converges in probability to zero as $N \rightarrow \infty$, which proves the lemma.
	
	\,
	
	It remains to prove the claim. By Lemma \ref{fluclemma2}, for each $0 \leq t \leq T$,
	\begin{equation}\label{d5}
	\limsup_{N \rightarrow \infty}\, \mathbb{E}\, \big[N^{-(d+2)} \sum_{x \in \mathbb{Z}^{d}} \eta^N_t (x)^2  H(x/N) \big] = 0,
	\end{equation}
	which implies that $N^{-(d+2)} \sum_{x \in \mathbb{Z}^{d}} \eta^N_t (x)^2  H(x/N)$ converges in probability to zero for each fixed $0 \leq t \leq T$. Therefore, to conclude the proof, we only need to prove that the sequence of the process
	$$
	\big\{ N^{-(d+2)} \sum_{x \in \mathbb{Z}^{d}} \eta^N_t (x)^2  H(x/N),\, 0 \leq t \leq T\big\}_{N \geq 1}
	$$
	is tight.  The proof is  similar as that given  in Subsection \ref{subsec tight}.  We need the fourth moment to be uniformly bounded as given in  Proposition \ref{fourth moment}.  The familiar readers can skip it without much of loss.  We give the detailed proof of the tightness below for completeness.
	
	Now we prove the tightness by checking Aldous' two conditions.   By \eqref{d5}, Aldous' first condition \eqref{aldous1} is satisfied.  A similar computation as in Subsection \ref{subsec tight} shows that
	\begin{align*}
	N^{-(d+2)} \sum_{x \in \mathbb{Z}^{d}} \eta^N_t (x)^2  H(x/N)
	- N^{-(d+2)} \sum_{x \in \mathbb{Z}^{d}} \eta^N_0 (x)^2  H(x/N)
	- \int_0^t\,   N^{-d} \sum_{x \in \mathbb{Z}^{d}}\, H (x/N)  A_s (x,\eta^N)  \,ds
	\end{align*}
	is a martingale with quadratic variation given by
	\begin{align*}
	\int_{0}^t\,    N^{-2(d+1)} \sum_{x \in \mathbb{Z}^{d}}\, H(x/N)^2 B_s (x,\eta^N)   \,ds
	\end{align*}
	where
	\begin{align*}
	A_s (x,\eta) &= (1 - 4 \lambda d) \eta_s (x)^2 + \lambda \sum_{y \sim x}\, \big\{ \eta_s (y)^2 + 2 \eta_s (x) \eta_s (y) \big\},\\
	B_s (x,\eta) &= \eta_s (x)^4 + \lambda \sum_{y \sim x} \big(2\eta_s (x) \eta_s (y) + \eta_s (y)^2\big)^2.
	\end{align*}
	Therefore, for any stopping time $\tau$ bounded from above by $T$ and for any $\theta \leq \gamma$, by Cauchy-Schwartz inequality,
	\begin{align}
	&\mathbb{E} \Big[  \big(\int_\tau^{\tau+\theta}\,   N^{-d} \sum_{x \in \mathbb{Z}^{d}}\, H (x/N)  A_s (x,\eta^N)  \,ds\big)^2 \Big]\notag\\
	&\quad \leq \gamma \int_0^{T+\gamma}\,   \mathbb{E} \Big[\big(N^{-d} \sum_{x \in \mathbb{Z}^{d}}\, H (x/N)  A_s (x,\eta^N)\big)^2 \Big] \,ds \notag\\
	&\quad \leq  \gamma \int_{0}^{T + \gamma}\, \big(N^{-d} \sum_{x \in \mathbb{Z}^{d}} H (x/N) \big)\,\big(N^{-d} \sum_{x \in \mathbb{Z}^{d}} H (x/N) \mathbb{E} [A_s (x,\eta^N)^2]\big)\, ds. \label{d6}
	\end{align}
	By Proposition \ref{fourth moment}, $ \mathbb{E} [A_s (x,\eta^N)^2]$ is uniformly bounded in $N$ and $t$. Thus, the term \eqref{d6} converges to zero as $N \rightarrow \infty$ and then $\gamma \rightarrow 0$. This shows that the integral term satisfies Aldous' second condition \eqref{aldous2}.  For the martingale term, first note that  $ \mathbb{E} [B_s (x,\eta^N)]$ is uniformly bounded in $N$ and $t$ according to Proposition \ref{fourth moment}. Therefore, the quadratic variation of the martingale vanishes in the limit $N \rightarrow \infty$. Then it is easy to see that the martingale term also satisfies \eqref{aldous2}.  This completes the proof.
\end{proof}

\begin{lemma}\label{lem2}
	If  $d$ and  $\lambda$ are large enough, then for any function $H \in \mathcal{S} (\mathbb{R}^d)$, the initial value satisfies $\mathcal{Y}_0 (H) \equiv 0$. Moreover,  for any function $G \in C_c^\infty (\mathbb{R})$,
	\begin{align*}
	G (\mathcal{Y}_t (H)) - G (\mathcal{Y}_0 (H))
	- \lambda \int_{ 0 }^t  G' (\mathcal{Y}_s (H))    \mathcal{Y}_s (\Delta H) \,ds
	- \frac{1}{2}||H||_2^2 \,C(\lambda,d) \int_{ 0 }^t G'' ( \mathcal{Y}_s (H))\,ds
	\end{align*}
	is a martingale with respect to the natural filtration $\mathcal{F}_t = \sigma (\mathcal{Y}_s (H), s \leq t, H \in \mathcal{S} (\mathbb{R}^d))$.
\end{lemma}

The above lemma gives a martingale characterization of the limit. By Remark \ref{rm3}, the limit is uniquely determined.

\begin{proof}[Proof of Lemma \ref{lem2}]
	For any function $H \in \mathcal{S} (\mathbb{R}^d)$,  since initially $\{\eta^N_0 (x)\}_{x \in \mathbb{Z}^d}$ are i.i.d. and  the second moment is uniformly bounded, it is easy to check that ${\mathrm{Var}} (\mathcal{Y}_0^N (H))$ is of order $\mathcal{O} (N^{-2})$. Since $\mathcal{Y}_0^N (H)$ has mean zero, $\mathcal{Y}_0^N (H)$ converges to zero in probability as $N \rightarrow \infty$. Thus, $\mathcal{Y}_0 (H) \equiv 0$. Now fix $G \in C_c^\infty (\mathbb{R})$.  By Dynkin's martingale,
	\begin{equation}\label{m1}
	\mathsf{M}^N_t  := \mathsf{M}^N_t (G,H) := G (\mathcal{Y}_t^N (H)) - G (\mathcal{Y}_0^N (H)) - \int_{ 0 }^t\, N^2 \mathcal{L}_N G (\mathcal{Y}_s^N (H))\, ds
	\end{equation}
	is a martingale with respect to the natural filtration. Define $\mathsf{M}_t$ as
	\begin{align*}
	&\mathsf{M}_t:= G (\mathcal{Y}_t (H)) - G (\mathcal{Y}_0 (H))
	- \lambda \int_{ 0 }^t  G' (\mathcal{Y}_s (H))    \mathcal{Y}_s (\Delta H) \,ds\\
	& \qquad - \frac{1}{2}||H||_2^2 \,C(\lambda,d) \int_{ 0 }^t G'' ( \mathcal{Y}_s (H))\,ds,
	\end{align*}
	where $\{\mathcal{Y}_t\}$ is the limit of $\{\mathcal{Y}_t^N\}$ as we have mentioned above.
	
	To conclude the proof of the lemma, we only need to show that $\{\mathsf{M}_t\}_{t\geq 0}$ is a martingale. By \cite[Theorem 5.3]{whitt2007proofs}, this property will hold if we show that $\mathsf{M}^N_t$ converges to $\mathsf{M}_t$ in distribution for each fixed $0 \leq t \leq T$ and that $\{\mathsf{M}^N_t\}_{N\geq 1}$ are uniformly integrable for each $t\geq 0$.
	
	Direct calculations yield that
	\begin{align*}
	&N^2 \mathcal{L}_N G (\mathcal{Y}_s^N (H)) = N^2 \sum_{x \in \mathbb{Z}^{d}}\, \Big[G \big(\mathcal{Y}_s^N (H) - N^{- (1+d/2)} \eta^N_s (x) H (x/N)\big) - G (\mathcal{Y}_s^N (H))\Big] \\
	&\quad + \lambda N^2 \sum_{x \in \mathbb{Z}^{d}}\,\sum_{y \sim x}\, \Big[G \big(\mathcal{Y}_s^N (H) + N^{- (1+d/2)} \eta^N_s (y) H (x/N)\big) - G (\mathcal{Y}_s^N (H))\Big] \\
	&\quad + (1 - 2 \lambda d) N^2 \sum_{x \in \mathbb{Z}^{d}}\,
	G'  (\mathcal{Y}_s^N (H)) N^{-(1+d/2)} \eta^N_s (x) H(x/N).
	\end{align*}
	Using Taylor's expansion and rearranging the above terms,  we rewrite $N^2 \mathcal{L}_N G (\mathcal{Y}_s^N (H)) $ as
	\begin{equation}\label{d1}
	\lambda G' (\mathcal{Y}_s^N (H)) \mathcal{Y}_s^N (\Delta_N H)
	+ \frac{1}{2N^d}   \sum_{x \in \mathbb{Z}^{d}} G'' (\mathcal{Y}_s^N (H)) \eta^N_s (x)^2 \big( H(x/N)^2 + \lambda \sum_{y \sim x} H(y/N)^2 \big)
	\end{equation}
	plus an error term bounded from above by
	\begin{equation*}
	\frac{||G'''||_\infty }{6N^{1+3d/2}} \sum_{x \in \mathbb{Z}^{d}} H(x/N)^3 \big\{ \eta_s^N (x)^3 + \lambda \sum_{y \sim x} \eta_s^N(y)^3 \big\}.
	\end{equation*}
	From Proposition  \ref{fourth moment}, the above term converges in $L^1$ to zero as $N \rightarrow \infty$.
	 By Lemma \ref{lemma 4.3} and Proposition
	\ref{fourth moment}, the second moment of the term \eqref{d1} is uniformly bounded in $N$, thus the martingale $\mathsf{M}^N_t$ is uniformly integrable.
	
	To prove that $\mathsf{M}^N_t$ converges to $\mathsf{M}_t$ in distribution, first note that we can replace the discrete Laplacian $\Delta_N$ by continuous Laplacian $\Delta$ in the first term of \eqref{d1}.  Let
	\begin{align*}
	\bar{H}^2 (x/N) :=  H(x/N)^2 + \lambda \sum_{y \sim x} H(y/N)^2.
	\end{align*}
	To deal with the second term in \eqref{d1}, we first replace it by
	\begin{equation}\label{d7}
	\frac{1}{2N^d}   \sum_{x \in \mathbb{Z}^{d}} G'' (\mathcal{Y}_s^N (H)) \mathbb{E} [\eta^N_s (x)^2] \bar{H}^2 (x/N)
	\end{equation}
	since the second moment of the error term is bounded from above by
	\begin{align*}
	&\frac{||G''||_\infty^2 }{4 N^{2d}} \sum_{x,\,y\in \mathbb{Z}^{d}} \mathrm{Var} \big(\eta^N_s (x)^2,\eta^N_s (y)^2 \big) \bar{H}^2 (x/N) \bar{H}^2 (y/N)\\
	&\leq C\, \big(\sup_{|x-y| > N^{1/2}}\, \mathrm{Var} \big(\eta^N_s (x)^2,\eta^N_s (y)^2 \big)\big)\,
	\frac{||G''||_\infty^2 }{N^{2d}} \sum_{x,\,y\in \mathbb{Z}^{d}}  \bar{H}^2 (x/N) \bar{H}^2 (y/N)\\
	&+ C\,\big(\sup_{0 \leq t \leq T,\, N \geq 1} \mathbb{E} [\eta^N_t (x)^4]\big)\, \frac{||G''\bar{H}||_\infty^2 }{ N^{3d/2}} \sum_{x \in \mathbb{Z}^{d}} \bar{H}^2 (x/N) ,
	\end{align*}
	which converges to zero as $N \rightarrow \infty$ by Propositions \ref{fourth moment} and \ref{correlation}. Next, according to Lemma \ref{fluclemma2} we could replace \eqref{d7} by
	\begin{equation*}
	\frac{1+1/h_\lambda}{2N^d}   \sum_{x \in \mathbb{Z}^{d}} G'' (\mathcal{Y}_s^N (H))  \bar{H}^2 (x/N).
	\end{equation*}
	Finally, without much effort, we replace the above term by $(1/2) C (\lambda,d) ||H||_2^2  G'' (\mathcal{Y}_s^N (H)) $. In conclusion, we have shown that $\mathsf{M}^N_t$ equals
	\begin{equation*}
	G(\mathcal{Y}^N_t (H)) - G(\mathcal{Y}^N_0 (H)) - \int_{0}^t\, \lambda G' (\mathcal{Y}_s^N (H)) \mathcal{Y}_s^N (\Delta H) + \frac{1}{2} C (\lambda,d) ||H||_2^2  G'' (\mathcal{Y}_s^N (H)) \, ds
	\end{equation*}
	plus an error term which converges to zero in probability. Since the process $\{\mathcal{Y}^N_t\}$ converges to $\{\mathcal{Y}_t\}$ in distribution, the martingale $\mathsf{M}^N_t$ converges in distribution to $\mathsf{M}_t$, and the proof is completed.
\end{proof}

\section{Proof of Proposition \ref{fourth moment}}\label{sec fourm}

In this section, we will prove Proposition \ref{fourth moment}. To make our strategy of the proof easy to catch, we first in Subsection \ref{subsec secm}  revisit a lemma given in \cite{Gri1983} by Griffeath, which says that the second moment of the occupation variable is uniformly bounded. Then, we prove the uniform boundedness of the fourth moment in the rest subsections. Note that, throughout this section, $\{\eta_t\}_{t\geq 0}$ is the Markov process  with generator $\mathcal{L}_N$ given in \eqref{equ 1.1 generator} with $\lambda_1=\lambda_2=0$.

\subsection{A preliminary  illustration for second moments.}\label{subsec secm}

In this subsection, we reprove the following lemma which was already proved in \cite{Gri1983}.

\begin{lemma}\label{second moments}(Griffeath, 1983)
If $d\geq3, \lambda>\frac{1}{2d(2\gamma_d-1)}$ and $\eta_0(x)=1$ for all $x\in \mathbb{Z}^d$, then
\[
\sup_{t\geq 0}{\rm E}\left[\eta_t(O)^2\right]<+\infty.
\]
\end{lemma}

\proof[An alternative  proof of Lemma \ref{second moments}]

As we have mentioned in the proof of Lemma \ref{fluclemma2}, when $\eta_0(x)=1$ for all $x$,
\begin{align*}
\mathbb{E}\left[\eta_t(O)^2\right]=\sum_{x\in \mathbb{Z}^d}e^{t\mathbf{\Psi}}(O, x).
\end{align*}
Let $\Upsilon=\mathbf{\Psi}+4\lambda d \mathbf{I}_{\mathbb{Z}^d\times \mathbb{Z}^d}$, where $\mathbf{I}_{\mathbb{Z}^d\times \mathbb{Z}^d}$ is the $\mathbb{Z}^d\times \mathbb{Z}^d$ identity matrix, then
\begin{align*}
{\mathbb{E}}\left[\eta_t(O)^2\right]
&=e^{-4\lambda dt}\sum_{n=0}^{+\infty}\frac{t^n}{n!}\sum_{x\in \mathbb{Z}^d}\Upsilon^n(O, x)\\
&=e^{-4\lambda dt}\sum_{n=0}^{+\infty}\frac{t^n}{n!}\sum_{x_0, x_1,...,x_n\in \mathbb{Z}^d: \atop x_0=O}\prod_{i=0}^{n-1}\Upsilon(x_i,x_{i+1}).
\end{align*}
For any $x\in \mathbb{Z}^d$, we define $\mathfrak{m}(x)=2d$ if $x\neq O$ and $\mathfrak{m}(O)=2d+1$. Furthermore, we introduce the random walk $\{\beta_n\}_{n\geq 1}$ on $\mathbb{Z}^d$ such that, for each $n\geq 0$,
\[
P\big(\beta_{n+1}=y\big|\beta_n=x\big)=\frac{1}{2d}
\]
if $x\neq O$ and $y$ is a neighbor of $x$, while
\[
P\big(\beta_{n+1}=y\big|\beta_n=O\big)=\frac{1}{2d+1}
\]
if $y$ is a neighbor of $O$ or $y=O$. For any $x, y\in \mathbb{Z}^d$, we define $\mathcal{H}(x,y)=\frac{\Upsilon(x,y)\mathfrak{m}(x)}{4d\lambda}$.
Then, according to the definition of $\Upsilon$,
\begin{equation*}
\mathbb{E}\left[\eta_t(O)^2\right]=e^{-4\lambda dt}\sum_{n=0}^{+\infty}\frac{(4\lambda dt)^n}{n!}{\rm E}_O\left[\prod_{i=0}^{n-1}\mathcal{H}(\beta_i, \beta_{i+1})\right],
\end{equation*}
where $\mathrm{E}_O$ is the expectation corresponding to the law of the random walk $\{\beta_n\}$ starting from $O$.
By direct calculation,
\[
\mathcal{H}(\beta_i, \beta_{i+1})=
\begin{cases}
1 & \text{~if\quad} \beta_i\neq O,\\
\frac{2d+1}{2d} & \text{~if\quad} \beta_i=O\text{~and~}\beta_{i+1}\neq O,\\
\frac{(1+2\lambda d)(2d+1)}{4\lambda d} & \text{~if\quad} \beta_i=\beta_{i+1}=O.
\end{cases}
\]
Hence $\mathcal{H}(\beta_i, \beta_{i+1})\geq 1$ and
\begin{equation}\label{equ secondmoments representation}
\mathbb{E} \left[\eta_t(O)^2\right]\leq e^{-4\lambda dt}\sum_{n=0}^{+\infty}\frac{(4\lambda dt)^n}{n!}{\rm E}_O\left[\prod_{i=0}^{+\infty}\mathcal{H}(\beta_i, \beta_{i+1})\right]={\rm E}_O\left[\prod_{i=0}^{+\infty}\mathcal{H}(\beta_i, \beta_{i+1})\right].
\end{equation}
Let $\hat{\tau}=\inf\{n\geq 1: \beta_n=O\}$, then
\[
P\big(\hat{\tau}<+\infty\big|\beta_0\sim O\big)=1-\gamma_d.
\]
 Since $\mathcal{H}(\beta_i, \beta_{i+1})=1$ when $\beta_i\neq O$,
\begin{align*}
&{\rm E}_O\left[\prod_{i=0}^{\hat{\tau}-1}\mathcal{H}(\beta_i, \beta_{i+1}), \hat{\tau}<+\infty\right]\\
&={\rm E}_O\left[\prod_{i=0}^{\hat{\tau}-1}\mathcal{H}(\beta_i, \beta_{i+1}), \hat{\tau}=1\right]+{\rm E}_O\left[\prod_{i=0}^{\hat{\tau}-1}\mathcal{H}(\beta_i, \beta_{i+1}), 1<\hat{\tau}<+\infty\right]\\
&=\frac{1}{2d+1}\frac{(2d+1)(1+2\lambda d)}{4\lambda d}+\frac{2d}{2d+1}\frac{2d+1}{2d}(1-\gamma_d)=\frac{1+2\lambda d}{4\lambda d}+1-\gamma_d.
\end{align*}
Then, by strong Markov property,
\[
{\rm E}_O\left[\prod_{i=0}^{+\infty}\mathcal{H}(\beta_i, \beta_{i+1})\right]
=\sum_{k=0}^{+\infty}\left(\frac{1+2\lambda d}{4\lambda d}+1-\gamma_d\right)^k\gamma_d.
\]
When $\lambda>\frac{1}{2d(2\gamma_d-1)}$, $\frac{1+2\lambda d}{4\lambda d}+1-\gamma_d<1$ and hence the above formula is finite.  This proves the  lemma by Eq. \eqref{equ secondmoments representation}.
\qed

\,

Liggett recalls Lemma \ref{second moments} in \cite[Section 9.6]{liggettips} and proves it in a different way with that given in \cite{Gri1983}. Our above proof is inspired by that given by Liggett a lot. The main difference between our proof and Liggett's is that Liggett bounds $e^{t\mathbf{\Psi}}(O, x)$ from above by finding an positive eigenvector of $\mathbf{\Psi}$ with respect to eigenvalue $0$ while we give up this approach, for the reason that such an eigenvector is difficult to calculate when we deal with the fourth moment  later.

Our proof of Proposition \ref{fourth moment} follows a similar analysis with that in the above proof of Lemma \ref{second moments}, where a critical step is to bound elements of $e^{t\mathbf{G}}$ from above, where $\mathbf{G}$ is a $(\mathbb{Z}^d)^4\times (\mathbb{Z}^d)^4$ matrix and will be given in the next subsection. To complete this step, a random walk on $(\mathbb{Z}^d)^4$ will be introduced. For mathematical details, see later subsections.

\subsection{Preparation.}  Hereafter, we shall prove Proposition \ref{fourth moment}. Before proving the fourth moment is uniformly bounded, we first introduce some notation. According to similarities among the four coordinates of the point $(x,y,z,w) \in (\mathbb{Z}^d)^4$, we divide the points on $(\mathbb{Z}^d)^4$ into the following five types:
\begin{enumerate}[(i)]
	\item if a point has the form $(x,x,x,x)$ where $x \in \mathbb{Z}^d$, then we say the point is of {\it type I};
	
	\item if a point has the form $(x,x,x,y),\,(x,x,y,x),\,(x,y,x,x)$ or $(y,x,x,x)$, where $x \neq y$ and $x,\,y \in \mathbb{Z}^d$, then we say the point is of {\it type II};
	
	\item if a point has the form $(x,x,y,y),\,(x,y,x,y)$ or $(x,y,y,x)$, where $x \neq y$ and $x,\,y \in \mathbb{Z}^d$, then we say the point is of {\it type III};
	
	\item if a point has the form $(x,x,y,z),\,(x,y,x,z),\,(x,y,z,x),\,(y,x,x,z),\,(y,x,z,x)$ or $(y,z,x,x)$, where $x,\,y,\,z$ are pairwise different and $x,\,y,\,z \in \mathbb{Z}^d$, then we say the point is of {\it type IV};
	
		\item if a point has the form $(x,y,z,w)$, where $x,\,y,\,z,\,w$ are pairwise different and $x,\,y,\,z,\,w \in \mathbb{Z}^d$, then we say the point is of {\it type V}.
\end{enumerate}
 We call the points of type I to type IV {\it bad points}, and call those of type V {\it good points}. Denote by $\mathcal{T} (\mathbf{x})$ the type of a point $\mathbf{x} \in (\mathbb{Z}^d)^4$.  More precisely, if $\mathbf{x}$ is of type I, then we set $\mathcal{T} (\mathbf{x}) = 1$, and the remaining four cases are similar. Let
\begin{equation*}
\mathbf{F}_t (x,y,z,w) := \mathbb{E}\, [\eta_t (x)\eta_t (y)\eta_t (z)\eta_t (w)], \quad \forall x,y,z,w \in \mathbb{Z}^d.
\end{equation*}
Then it can be checked directly that
\begin{equation}\label{a1}
\frac{d}{dt} \mathbf{F}_t (x,y,z,w) = (\mathbf{G} - 8 \lambda d \mathbf{I})\, \mathbf{F}_t (x,y,z,w),
\end{equation}
where $\mathbf{I}$ is the $(\mathbb{Z}^d)^4 \times (\mathbb{Z}^d)^4$ identical matrix, and the detailed form of the matrix $\mathbf{G}$ will be given later.

\,

Equation \eqref{a1} can be considered as an execution of the calculation
\begin{equation}\label{Hille-Yosida}
\frac{d}{dt}T_N (t)f(\eta)=T_N (t)\mathcal{L}_N f(\eta),
\end{equation}
where $\{T_N(t)\}_{t\geq 0}$ and $\mathcal{L}_N$ are respectively semi-group and generator of our model with parameters $\lambda_1=\lambda_2=0$ and $f(\eta) =
\eta(x)\eta(y)\eta(z)\eta(w)$. Such a calculation is rigorous according to \cite[Theorem 9.3.1]{liggettips}. Readers familiar with the theory of linear systems may point out that \cite[Theorem 9.3.1]{liggettips} only makes \eqref{Hille-Yosida} hold for $f$ with the form $f(\eta)=\eta(x)\eta(y)$, i.e., second moments instead of four moments. However, we can consider $\{\eta_t(x)\eta_t(y):~x\in \mathbb{Z}^d, y\in \mathbb{Z}^d, t\geq 0\}$ as a linear system with state space
$
[0, +\infty)^{\mathbb{Z}^d\times \mathbb{Z}^d}=[0, +\infty)^{\mathbb{Z}^{2d}}.
$
When we utilize \cite[Theorem 9.3.1]{liggettips} on such a linear system, we obtain exactly Equation \eqref{a1}.

\,

Now we give the explicit form of  $\mathbf{G}$,  which depends on the types of the point $(x,y,z,w)$.
\begin{enumerate}[(i)]
	\item The point $(x,y,z,w)$ is of type I, i.e., $(x,y,z,w) = (x,x,x,x)$ for some $x \in \mathbb{Z}^d$. Then
	$$
	\mathbf{G} ((x,x,x,x),(x,x,x,x)) = 3,
	$$
	and
	$$
	\mathbf{G} ((x,x,x,x),(x',y',z',w')) = \lambda
	$$
	 if $(x',y',z',w')$ equals
	 $$(y,y,y,y),$$
	 or
	 $$
	 (y,x,x,x), \,(x,y,x,x),\, (x,x,y,x),\, (x,x,x,y),
	 $$
	 or
	 $$(y,y,x,x),\,(y,x,y,x),\,(y,x,x,y),\,(x,y,y,x),\,(x,y,x,y),\,(x,x,y,y),$$
	 or
	 $$(y,y,y,x),\,(y,y,x,y),\,(y,x,y,y),\,(x,y,y,y),$$
	  where $y \sim x$, and $\mathbf{G} ((x,x,x,x),(x',y',z',w')) = 0$ otherwise.  Therefore,
	  \begin{equation*}
	  |\{(x',y',z',w'):   \mathbf{G} ((x,x,x,x),(x',y',z',w')) \neq 0 \}| = 30d+1.
	  \end{equation*}
	
	  \item The point $(x,y,z,w)$ is of type II.  Let us take $(x,y,z,w) = (x,x,x,y)$, $y \neq x$ as an example.  Then
	  $$
	  \mathbf{G} ((x,x,x,y),(x,x,x,y)) = 2,
	  $$
	  and
	  $$
	  \mathbf{G} ((x,x,x,y),(x',y',z',w')) = \lambda
	  $$
	  if $(x',y',z',w')$ equals
	  $$(x,x,x,u), \, u \sim y \qquad \text{or} \qquad (v,v,v,y), \, v \sim x,$$
	  or
	  $$
	  (u,u,x,y),\,(u,x,u,y),\,(x,u,u,y),\,u \sim x,
	  $$
	  or
	  $$
	  (u,x,x,y),\,(x,u,x,y),\,(x,x,u,y),\,u \sim x,
	  $$
	  and $\mathbf{G} ((x,x,x,y),(x',y',z',w')) = 0$ otherwise. The expression for other  type II points is similar.  Therefore,
	  \begin{equation*}
	  |\{(x',y',z',w'):   \mathbf{G} ((x,x,x,y),(x',y',z',w')) \neq 0 \}| = 16d+1.
	  \end{equation*}

	   \item The point $(x,y,z,w)$ is of type III.  Let us take $(x,y,z,w) = (x,x,y,y)$, $y \neq x$ as an example.  Then
	  $$
	  \mathbf{G} ((x,x,y,y),(x,x,y,y)) = 2,
	  $$
	  and
	  $$
	  \mathbf{G} ((x,x,y,y),(x',y',z',w')) = \lambda
	  $$
	  if $(x',y',z',w')$ equals
	  $$
	  (u,u,y,y),\,(u,x,y,y),\,(x,u,y,y),\,u \sim x,
	  $$
	  or
	  $$
	  (x,x,v,v),\,(x,x,v,y),\,(x,x,y,v),\,v \sim y,
	  $$
	  and $\mathbf{G} ((x,x,y,y),(x',y',z',w')) = 0$ otherwise. The expression for other  type III points is similar.  Therefore,
	  \begin{equation*}
	  |\{(x',y',z',w'):   \mathbf{G} ((x,x,y,y),(x',y',z',w')) \neq 0 \}| = 12d+1.
	  \end{equation*}

	  \item The point $(x,y,z,w)$ is of type IV.  Let us take $(x,y,z,w) = (x,x,y,z)$  as an example, where $x,y,z$ are pairwise different.  Then
	  $$
	  \mathbf{G} ((x,x,y,z),(x,x,y,z)) = 1,
	  $$
	  and
	  $$
	  \mathbf{G} ((x,x,y,z),(x',y',z',w')) = \lambda
	  $$
	  if $(x',y',z',w')$ equals
	  $$
	  (x,x,u,z),\,u \sim y \qquad \text{or} \qquad  (x,x,y,v),\,v \sim z,
	  $$
	  or
	  $$
	  (u,u,y,z),\,(u,x,y,z),\,(x,u,y,z),\,u \sim x,
	  $$
	  and $\mathbf{G} ((x,x,y,z),(x',y',z',w')) = 0$ otherwise. The expression for other   type IV points is similar. Therefore,
	  \begin{equation*}
	  |\{(x',y',z',w'):   \mathbf{G} ((x,x,y,z),(x',y',z',w')) \neq 0 \}| = 10d+1.
	  \end{equation*}

	  \item The point $(x,y,z,w)$ is of type V, i.e.,  $x,y,z,w$ are pairwise different.  Then
	  $$
	  \mathbf{G} ((x,y,z,w),(x',y',z',w')) = \lambda
	  $$
	  if $(x',y',z',w')$ equals
	  $$
	  (u,y,z,w),\,u \sim x \quad \text{or} \quad  (x,u,z,w),\,u \sim y \quad \text{or} \quad  (x,y,u,w),\,u \sim z
	  \quad \text{or} \quad  (x,y,z,u),\,u \sim w,
	  $$
	  and $\mathbf{G} ((x,y,z,w),(x',y',z',w')) = 0$ otherwise.  Therefore,
	  \begin{equation*}
	  |\{(x',y',z',w'):   \mathbf{G} ((x,y,z,w),(x',y',z',w')) \neq 0 \}| = 8d.
	  \end{equation*}
\end{enumerate}
Next we introduce a random walk $\{S_n\}_{n \geq 0}$ on $(\mathbb{Z}^d)^4$. For $\mathbf{x} \in (\mathbb{Z}^d)^4$,  let
\begin{equation*}
M (\mathbf{x}) =
\begin{cases}
30d+1 &\text{if $\mathbf{x}$ is type I},\\
16d+1 &\text{if $\mathbf{x}$ is type II},\\
12d+1 &\text{if $\mathbf{x}$ is type III},\\
10d+1 &\text{if $\mathbf{x}$ is type IV},\\
8d &\text{if $\mathbf{x}$ is type V}.
\end{cases}
\end{equation*} The transition probability $p (\cdot,\cdot)$ of the random walk $(S_n)_{n \geq 0}$ is given by
\begin{equation*}
p(\mathbf{x},\mathbf{y}) =  1/M(\mathbf{x}) \quad \text{if \quad $G (\mathbf{x},\mathbf{y}) \neq 0$},
\end{equation*}
and $p(\mathbf{x},\mathbf{y}) = 0$ otherwise.

\begin{remark}
	If $\lambda_1 \neq 0$, we need to introduce an asymmetric random walk. More precisely, if $\mathbf{G} (\mathbf{x},\mathbf{y}) \neq 0$, then
	\begin{equation*}
	p(\mathbf{x},\mathbf{y}) =
	\begin{cases}
	\frac{\mathbf{G} (\mathbf{x},\mathbf{y})}{\lambda M(\mathbf{x})} &\text{if $\mathbf{y} \neq \mathbf{x}$},\\
	1/M(\mathbf{x}) &\text{if $\mathbf{y} = \mathbf{x}$},
	\end{cases}
	\end{equation*}
and  if $\mathbf{G} (\mathbf{x},\mathbf{y}) = 0$, then $p(\mathbf{x},\mathbf{y}) = 0$.
\end{remark}

\subsection{Upper bound on the fourth moment.} Using \eqref{a1}, for any $x \in \mathbb{Z}^d$,
\begin{align*}
\mathbb{E} [\eta_t (x)^4] &= \mathbf{F}_t (x,x,x,x) = e^{t(\mathbf{G} - 8 \lambda d \mathbf{I})} \mathbf{F}_0 (x,x,x,x)\\
&= e^{-8 \lambda d t} \sum_{\mathbf{y} \in (\mathbb{Z}^d)^4} e^{t \mathbf{G}} ((x,x,x,x),\mathbf{y}) \mathbf{F}_0 (\mathbf{y}).
\end{align*}
Since $\sup_{\mathbf{y}\in (\mathbb{Z}^d)^4}\mathbf{F}_0 (\mathbf{y}) \leq C$ for some $C>0$ from the initial condition, the last formula is at most
\begin{align*}
&Ce^{-8 \lambda d t} \sum_{\mathbf{y} \in (\mathbb{Z}^d)^4} e^{t \mathbf{G}} ((x,x,x,x),\mathbf{y})
=Ce^{-8 \lambda d t} \sum_{\mathbf{y} \in (\mathbb{Z}^d)^4} \sum_{n=0}^{\infty} \frac{t^n}{n!} \mathbf{G}^n ((x,x,x,x), \mathbf{y})\\
&=Ce^{-8 \lambda d t}  \sum_{n=0}^{\infty} \frac{t^n}{n!}\, \sum_{\mathbf{x}_0,\mathbf{x}_1,\ldots,\mathbf{x}_n, \atop \mathbf{x}_0 = (x,x,x,x)} \,\prod_{i=0}^{n-1} \mathbf{G} (\mathbf{x}_i,\mathbf{x}_{i+1})\\
&=Ce^{-8 \lambda d t}  \sum_{n=0}^{\infty} \frac{(8\lambda d t)^n}{n!}\, \sum_{\mathbf{x}_0,\mathbf{x}_1,\ldots,\mathbf{x}_n, \atop \mathbf{x}_0 = (x,x,x,x)} \,\prod_{i=0}^{n-1} \frac{1}{M (\mathbf{x}_i)} \,\prod_{i=0}^{n-1} \frac{M (\mathbf{x}_i) \mathbf{G} (\mathbf{x}_i,\mathbf{x}_{i+1})}{8 \lambda d}.
\end{align*}
Note that the second summation above over $\mathbf{x}_0,\mathbf{x}_1,\ldots,\mathbf{x}_n$ such that $\mathbf{x}_0 = (x,x,x,x)$ can be rewritten as
\begin{equation*}
\mathrm{E}_{(x,x,x,x)} \left[ \prod_{i=0}^{n-1} \frac{M (S_i) \mathbf{G} (S_i,S_{i+1})}{8 \lambda d} \right],
\end{equation*}
which is bounded from above by
\begin{equation*}
\mathrm{E}_{(x,x,x,x)} \left[ \prod_{i=0}^{n-1} \mathbf{H} (S_i,S_{i+1}) \right],
\end{equation*}
where
\begin{equation*}
\mathbf{H} (\mathbf{x},\mathbf{y}) =
\begin{cases}
 \frac{M(\mathbf{x}) \mathbf{G} (\mathbf{x},\mathbf{y})}{8 \lambda d} \bigvee 1  &\text{if $\mathbf{x}$ is a bad point and $\mathbf{y} = \mathbf{x}$},\\
\frac{M(\mathbf{x}) \mathbf{G} (\mathbf{x},\mathbf{y})}{8 \lambda d} &\text{otherwise}.
\end{cases}
\end{equation*}
Note that $\mathbf{H} (S_i,S_{i+1}) \geq 1$ a.s. Therefore,
\begin{equation*}
\mathbb{E} [\eta_t (x)^4]  \leq C\mathrm{E}_{(x,x,x,x)} \left[ \prod_{i=0}^{\infty} \mathbf{H} (S_i,S_{i+1}) \right].
\end{equation*}

\subsection{Finiteness of the infinite-product expectation.} It remains to prove that
\begin{equation}\label{a6}
\mathrm{E}_{(x,x,x,x)} \left[ \prod_{i=0}^{\infty} \mathbf{H} (S_i,S_{i+1}) \right] < \infty.
\end{equation}
 If $\mathbf{x}$ is a good point, then
\begin{equation}\label{a5}
\mathrm{P}_{\mathbf{x}}\, (\text{$S_n$ is a bad point for some $n$}) \leq 6\, \Gamma_d,
\end{equation}
where $\Gamma_d = 1 - \gamma_d$ is the return probability of a simple random walk on $\mathbb{Z}^d$ starting from the origin. Note that the trajectories of the random walk $\{S_n\}_{n \geq 0}$ alternate between bad points and good points.  Moreover, if $d$ is large enough such that $6 \Gamma_d < 1$, then, with probability one,  there exists $n_0$ such that $S_n$ is a good point for all $n > n_0$. Since $\mathbf{H} (S_i,S_{i+1}) = 1$ a.s.\,if $S_i$ is a good point,  there are only finite terms strictly larger than one in the infinite product appearing in \eqref{a6}. Based on the above observations, it is easy to see that \eqref{a6} follows directly from the strong Markov  property and the following lemma.

\begin{lemma}\label{lemma 3.3}
There exist $d_0$ and $\lambda_0$ such that for all $d > d_0$ and for all $\lambda > \lambda_0$, there exists a constant $K = K(\lambda,d)$ such that $6 \,\Gamma_d\,K < 1$ and that
\begin{equation*}
\sup_{\text{$\mathbf{x}$ is a bad point}} \mathrm{E}_{\mathbf{x}} \Big[\,\prod_{i=0}^{\tau - 1} \mathbf{H} (S_i,S_{i+1})\,\Big] \leq K,
\end{equation*}
where
\begin{equation*}
\tau = \inf \{n > 0: S_n \,\, \text{is a good point}\}.
\end{equation*}
\end{lemma}

\begin{proof}[Proof of Lemma \ref{lemma 3.3}]
Let $\tau_0 = 0$ and define recursively $\tau_n,\, n \geq 1$ as
\begin{equation*}
\tau_n = \inf \{ m > \tau_{n-1}: \mathcal{T} (S_m) \neq \mathcal{T} (S_{\tau_{n-1}})\}.
\end{equation*}
Let
$$\mathcal{P} (i \rightarrow j) = \sup_{\mathbf{x}:\mathcal{T}(\mathbf{x})=i}\,\mathrm{P} (\,\mathcal{T} (S_{\tau_n}) = j\, |\, S_{\tau_{n-1}} = \mathbf{x}\,).$$
Then it can be checked directly that  if $d \geq 5$,
$$
\mathcal{P} (i \rightarrow i-1) \leq 2/d\quad \text{and} \quad  \mathcal{P} (i \rightarrow i-2) \leq 2/d  \quad \text{if \quad $d \geq 3$},
$$
and
$$
\mathcal{P} (2 \rightarrow 1) \leq 2/d.
$$
To see the property above, we take $\mathcal{T}(S_i)=2$ as an example. If $S_i=(x,x,x,y)$ for some $x\sim y$ and some $i\geq 1$, then
\begin{align*}
&\mathrm{P}\big(\mathcal{T}(S_{i+1})=1\big|\mathcal{T}(S_{i+1})\neq 2, S_i=(x,x,x,y)\big)\\
&=\mathrm{P}\big(S_{i+1}=(x,x,x,x)\text{~or~}(y,y,y,y)\big|\mathcal{T}(S_{i+1})\neq 2, S_i=(x,x,x,y)\big)\\
&=\frac{2}{12d-1}\leq \frac{2}{d}
\end{align*}
for  $d \geq 5$. Note that the above probability is $0$ if $y\neq x$ and $y$ is not a neighbor of $x$.
We remark that the constant $2$ is not so important. The main point is that the probability is of order $\mathcal{O} (1/d)$ that the type $\mathcal{T} (S_{\tau_n})$ becomes smaller and  such a transition occurs when $S_n$ has some coordinates that are neighbors with each other. Let $(Y_n)_{n \geq 0}$ be a $\{1,2,3,4,5\}$-valued Markov process with transition probability $p (\cdot,\cdot)$ given by
\begin{align*}
&\qquad p (1,2) = 1, \,\,p (2,1) = 2/d, \,\,p (2,3) = 1 - 2/d, \,\,p (5,4) = 1,\\
&p (i,i-1) = p(i,i-2) = 2/d \quad  \text{and} \quad p (i,i+1) = 1 - 4/d \quad \text{for $i = 3,\,4$}.
\end{align*}
Then we can couple $(\mathcal{T} (S_{\tau_n}))_{n \geq 0}$ and $(Y_n)_{n \geq 0}$ together such that
\begin{equation}\label{a7}
\mathcal{T} (S_{\tau_n}) \geq Y_n, \quad \text{for all $n \geq 0$}.
\end{equation}
Define
\begin{equation*}
\mathscr{T} = \inf \{n: \text{$S_{\tau_n}$ is a good point}\} \quad \text{and} \quad
\tilde{\mathscr{T}} = \inf \{n: Y_n = 5\}.
\end{equation*}
Then $\mathscr{T}  \leq \tilde{\mathscr{T}}$ by \eqref{a7}.

\,

\begin{claim}
There exists a constant $C_1 < \infty$ independent of $\lambda$ and $d$ such that if  $\lambda$ and $d$ are large enough, then
\begin{equation}\label{a9}
\sup_{\text{$\mathbf{x}$ is a bad point}} \mathrm{E}_{\mathbf{x}} \Big[\,\prod_{i=0}^{\tau_1 - 1} \mathbf{H} (S_i,S_{i+1})\,\Big] \leq C_1.
\end{equation}
\end{claim}

\,

Using the strong Markov property  and  \eqref{a9}, if  $\mathbf{x}$ is a bad point, then
\begin{equation*}
\mathrm{E}_{\mathbf{x}} \Big[\,\prod_{i=0}^{\tau - 1} \mathbf{H} (S_i,S_{i+1})\,\Big] \leq \mathrm{E}_{\mathbf{x}} [C_1^{\mathscr{T}}] \leq \mathrm{E}_1 [C_1^{\tilde{\mathscr{T}}}],
\end{equation*}
where $\mathrm{E}_1$ is the expectation of the random walk $(Y_n)_{n \geq 0}$ starting from $1$. Below we shall prove that
\begin{equation}\label{a8}
	\sup_{d > 324\,C_1^4} \, \mathrm{E}_1 [C_1^{\tilde{\mathscr{T}}}] < \infty.
\end{equation}
Then the lemma follows by taking the constant $K = 	\mathrm{E}_1 [C_1^{\tilde{\mathscr{T}}}]$ and from the well-known fact $\lim_{d \rightarrow \infty} \Gamma_d = 0$.

 It remains to prove \eqref{a9} and \eqref{a8}. We first prove \eqref{a9}. It can be checked directly that if $\lambda$ and $d$ are large enough, then there exists a constant $C_2 < 1$ independent of $\lambda$ and $d$ such that
 \begin{equation*}
  \sup_{\text{$\mathbf{x}$ is a bad point}} \mathrm{E}_{\mathbf{x}} [\mathbf{H} (\mathbf{x},S_1),\, \tau_1 > 1] < C_2.
 \end{equation*}
For example, if $\mathbf{x}$ is type I, then
\begin{equation*}
 \mathrm{E}_{\mathbf{x}} [\mathbf{H} (\mathbf{x},S_1),\, \tau_1 > 1]  \leq \frac{1}{30d+1} \bigvee \frac{3}{8 \lambda d}+  \frac{2\,\lambda\,d}{8\,\lambda\,d} < 1 \quad \text{for lage enough  $\lambda$ and $d$}.
\end{equation*}
The remaining three cases are similar, and we leave the details to the readers. By Markov property,
 \begin{equation*}
\mathrm{E}_{\mathbf{x}} \Big[\,\prod_{i=0}^{\tau_1 - 1} \mathbf{H} (S_i,S_{i+1})\,\Big] \leq \sum_{n = 1}^{\infty} 4\,C_2^{n-1} < \infty.
\end{equation*}
The constant $4$ comes from the fact that $||\mathbf{H}||_\infty \leq 4$. This proves \eqref{a9}. To prove \eqref{a8}, first note that if the random walk $\{Y_n\}$ has not hit the point $5$ before time $n$, then it must jump to left at least $n/4$ times in the first $n$ steps. Since each time the random walk has at most three choices,
\begin{equation*}
\mathrm{P}_1 (\tilde{\mathscr{T}} \geq n) \leq (2/d)^{n/4} 3^n.
\end{equation*}
Then
\begin{equation*}
 \mathrm{E}_1 [C_1^{\tilde{\mathscr{T}}}] \leq \sum_{n = 0}^{\infty} C_1^n \mathrm{P}_1 (\tilde{\mathscr{T}} \geq n)  < \infty
\end{equation*}
as long as $d > 324 \, C_1^4$. This proves \eqref{a8}, and thus finish the proof of the lemma.
\end{proof}

\begin{remark}\label{rm2}
Using the same strategy, we can show that, when $d$ and $\lambda$ are sufficiently large, there exists $\epsilon_0 > 0$ such that
\begin{equation}\label{a10}
\sup_{|x-y| > N^{1/2}}\,\mathrm{E}_{(x,x,y,y)} \left[ \prod_{i=0}^{\infty} \mathbf{H} (S_i,S_{i+1})^{1+\epsilon_0} \right] < \infty.
\end{equation}
This extension will be useful in the next section.
\end{remark}

\section{Proof of Proposition \ref{correlation}}\label{sec cor}

In this section we prove Proposition \ref{correlation}. Since the proof is similar to that of Proposition \ref{fourth moment}, we omit most of the details.

\,

\begin{proof}[Proof of Proposition \ref{correlation}]
Let
\begin{equation*}
\tilde{\mathbf{F}}_t (x,y,z,w) := \mathbb{E}\, [\eta_t (x)\eta_t (y)]\,\mathbb{E}\,[\eta_t (z)\eta_t (w)], \quad \forall x,y,z,w \in \mathbb{Z}^d.
\end{equation*}
As in the proof of Proposition \ref{fourth moment}, by \cite[Theorem 9.3.1]{liggettips},
\begin{equation*}
\frac{d}{dt} \tilde{\mathbf{F}}_t (x,y,z,w) = (\tilde{\mathbf{G}} - 8 \lambda d \mathbf{I})\, \tilde{\mathbf{F}}_t (x,y,z,w),
\end{equation*}
where $\mathbf{I}$ is the $(\mathbb{Z}^d)^4 \times (\mathbb{Z}^d)^4$ identical matrix and $\tilde{\mathbf{G}}$ is a $(\mathbb{Z}^d)^4 \times (\mathbb{Z}^d)^4$ matrix. For any $\mathbf{x}=(x_1, x_2, x_3, x_4)\in (\mathbb{Z}^d)^4$, by direct calculation, it is easy to check that $\tilde{\mathbf{G}}(\mathbf{x},\mathbf{y})\leq \mathbf{G}(\mathbf{x},\mathbf{y})$ for any $\mathbf{y}\in (\mathbb{Z}^d)^4$, which implies that
\[
\{\mathbf{y}: \tilde{\mathbf{G}}(\mathbf{x},\mathbf{y})>0\}\subseteq \{\mathbf{y}: \mathbf{G}(\mathbf{x},\mathbf{y})>0\}.
\]
This property allows us to use the same random walk $\{S_n\}_{n\geq 1}$ introduced in the proof of Proposition \ref{fourth moment} to study $\tilde{\mathbf{G}}$. Furthermore, by direct calculation, it is easy to check that
\[
 \mathbf{F}_0(\mathbf{x})=\tilde{\mathbf{F}}_0(\mathbf{x}) \quad \text{~and~} \quad  \tilde{\mathbf{G}}(\mathbf{x}, \mathbf{y})=\mathbf{G}(\mathbf{x}, \mathbf{y})
\]
for any $\mathbf{y}\in (\mathbb{Z}^d)^4$ when $\{x_1, x_2\} \cap \{x_3, x_4\} = \emptyset$ for $\mathbf{x}=(x_1, x_2, x_3, x_4)$.

In conclusion,
\begin{equation*}
\tilde{\mathbf{F}}_t (x,x,y,y) = e^{-8 \lambda d t}\, \sum_{n = 0}^\infty\, \frac{(8\lambda d t)^n}{n!} \mathrm{E}_{(x,x,y,y)}  \,\left[ \prod_{i=0}^{n-1} \frac{M(S_i) \tilde{\mathbf{G}} (S_i,S_{i+1})}{8 \lambda d}\tilde{\mathbf{F}}_0(S_n) \right].
\end{equation*}
Define
\begin{equation*}
\sigma  =  \inf \{n \geq 0: \{S_n (1), S_n (2)\} \cap \{S_n (3),S_n (4)\} \neq \emptyset\},
\end{equation*}
where $S_n (i)$, $1 \leq i \leq 4$,  is the $i$-th component of $S_n$. Then
\begin{align*}
&\mathrm{Cov} (\eta_t (x)^2, \eta_t(y)^2) = \mathbf{F}_t (x,x,y,y) - \tilde{\mathbf{F}}_t (x,x,y,y)\\
&= e^{-8 \lambda d t}\, \sum_{n = 0}^\infty\, \frac{(8\lambda d t)^n}{n!} \mathrm{E}_{(x,x,y,y)}  \,\left[ \prod_{i=0}^{n-1} \frac{M(S_i) {\mathbf{G}} (S_i,S_{i+1})}{8 \lambda d}\mathbf{F}_0(S_n) - \prod_{i=0}^{n-1} \frac{M(S_i) \tilde{\mathbf{G}} (S_i,S_{i+1})}{8 \lambda d} \tilde{\mathbf{F}}_0 (S_n) \right]\\
&= e^{-8 \lambda d t}\, \sum_{n = 0}^\infty\, \frac{(8\lambda d t)^n}{n!} \mathrm{E}_{(x,x,y,y)}  \,\Bigg[\prod_{i=0}^{n-1} \frac{M(S_i) {\mathbf{G}} (S_i,S_{i+1})}{8 \lambda d}\mathbf{F}_0(S_n)\\
&\text{\quad\quad\quad\quad}-\prod_{i=0}^{n-1} \frac{M(S_i) \tilde{\mathbf{G}} (S_i,S_{i+1})}{8 \lambda d} \tilde{\mathbf{F}}_0 (S_n), \sigma< \infty \Bigg],
\end{align*}
since $\tilde{\mathbf{G}} (S_i,S_{i+1})=\mathbf{G}(S_i,S_{i+1})$ and $\mathbf{F}_0(S_n)=\tilde{\mathbf{F}}_0 (S_n)$ on the event $\{\sigma=\infty\}$.
It is well-known that $\mathrm{Cov} (\eta_t (x)^2, \eta_t(y)^2) \geq 0$. Since $\sup_{\mathbf{y}\in (\mathbb{Z}^d)^4}\mathbf{F}_0(\mathbf{y})<+\infty$, to finish the proof of Proposition \ref{correlation}, we only need to prove that
\begin{equation}\label{b1}
\limsup_{N \rightarrow \infty}\, \sup_{|x-y| > N^{1/2}}\, \mathrm{E}_{(x,x,y,y)}  \,\left[ \prod_{i=0}^{n-1} \mathbf{H} (S_i,S_{i+1}), \sigma < \infty \right] = 0.
\end{equation}
Denote by $C_3$ the supremum in Remark \ref{rm2}. Then by H{\"o}lder's inequality, the supremum on the left-hand side of \eqref{b1} is bounded by
$$
C_3^{1/(1+\epsilon_0)}\,\sup_{|x-y| > N^{1/2}}\, \mathrm{P}_{(x,x,y,y)}\,(\sigma < \infty)^{1/(p_{\epsilon_0})},
$$
where $(1+\epsilon_0)^{-1} + p_{\epsilon_0}^{-1} = 1$. This completes the proof since
\begin{equation*}
\limsup_{N \rightarrow \infty}\, \sup_{|x-y| > N^{1/2}}\, \mathrm{P}_{(x,x,y,y)}\,(\sigma < \infty) = 0
\end{equation*}
when $d\geq 3$ according to the fact that $\{S_n(i)-S_n(j)\}_{n\geq 0}$ is a lazy version of the simple random walk on $\mathbb{Z}^d$ for $1\leq i\leq 2$ and $3\leq j\leq 4$.
\end{proof}

\quad

\textbf{Acknowledgments.}  Xue thanks the financial support from the National Natural Science Foundation of China with grant numbers 11501542.  
Zhao thanks the financial support from the ANR grant MICMOV (ANR-19-CE40-0012) of the French National Research Agency (ANR).

\bibliographystyle{plain}
\bibliography{zhaoreference}

\begin{thebibliography}{10}

\bibitem{erhard2020non}
D.~Erhard, T.~Franco, P.~Gon{\c{c}}alves, A.~Neumann, and M.~Tavares.
\newblock Non-equilibrium fluctuations for the {SSEP} with a slow bond.
\newblock In {\em Annales de l'Institut Henri Poincar{\'e}, Probabilit{\'e}s et
  Statistiques}, volume~56, pages 1099--1128. Institut Henri Poincar{\'e},
  2020.

\bibitem{Gri1983}
D.~Griffeath.
\newblock The binary contact path process.
\newblock {\em The Annals of Probability}, 11(3):692--705, 1983.

\bibitem{HolleyStroock78}
R.~A. Holley and D.~W. Stroock.
\newblock Generalized {Ornstein-Uhlenbeck} processes and infinite particle
  branching brownian motions.
\newblock {\em Publications of the Research Institute for Mathematical
  Sciences}, 14(3):741--788, 1978.

\bibitem{jara2020stochastic}
M.~Jara and C.~Landim.
\newblock The stochastic heat equation as the limit of a stirring dynamics
  perturbed by a voter model.
\newblock {\em arXiv preprint arXiv:2008.03076}, 2020.

\bibitem{jaram18nonequilireaction}
M.~Jara and O.~Menezes.
\newblock Non-equilibrium fluctuations for a reaction-diffusion model via
  relative entropy.
\newblock {\em arXiv preprint arXiv:1810.03418}, 2018.

\bibitem{jara2018non}
M.~Jara and O.~Menezes.
\newblock Non-equilibrium fluctuations of interacting particle systems.
\newblock {\em arXiv preprint arXiv:1810.09526}, 2018.

\bibitem{klscaling}
C.~Kipnis and C.~Landim.
\newblock {\em Scaling limits of interacting particle systems}, volume 320.
\newblock Springer Science \& Business Media, 2013.

\bibitem{Lawler2010}
G.~F. Lawler and V.~Limic.
\newblock {\em Random walk: a modern introduction}, volume 123.
\newblock Cambridge University Press, 2010.

\bibitem{liggettips}
T.~M. Liggett.
\newblock {\em Interacting particle systems}, volume 276.
\newblock Springer Science \& Business Media, 2012.

\bibitem{mitoma1983tightness}
I.~Mitoma.
\newblock Tightness of probabilities on ${C} ([0,1];\mathscr{Y}')$ and ${D}
  ([0,1];\mathscr{Y}')$.
\newblock {\em The Annals of Probability}, 11(4):989--999, 1983.

\bibitem{PresuttiSpohn83}
E.~Presutti and H.~Spohn.
\newblock Hydrodynamics of the voter model.
\newblock {\em The Annals of Probability}, 11(4):867--875, 1983.

\bibitem{ravishankar1992fluctuations}
Krishnamurthi Ravishankar.
\newblock Fluctuations from the hydrodynamical limit for the symmetric simple
  exclusion in $\mathbb{Z}^d$.
\newblock {\em Stochastic processes and their applications}, 42(1):31--37,
  1992.

\bibitem{whitt2007proofs}
W.~Whitt.
\newblock Proofs of the martingale {FCLT}.
\newblock {\em Probability Surveys}, 4:268--302, 2007.

\bibitem{xue2020hydrodynamics}
X.F. Xue and L.J. Zhao.
\newblock Hydrodynamics of the weakly asymmetric normalized binary contact path
  process.
\newblock {\em Stochastic Processes and their Applications}, 2020.

\bibitem{yau1991relative}
H.~T. Yau.
\newblock Relative entropy and hydrodynamics of {Ginzburg-Landau} models.
\newblock {\em Letters in Mathematical Physics}, 22(1):63--80, 1991.

\end{thebibliography}

\end{document}